\documentclass[12pt,reqno]{article}
\usepackage{amsfonts, amsthm, amsmath, amssymb}
\usepackage{caption,array,verbatim,enumitem}
\usepackage[T1]{fontenc}
\usepackage{tikz}
\usepackage{times}
\usepackage{geometry}
\usepackage{bbold}
\usepackage{setspace}
\usepackage{breakcites}
\newcommand*\samethanks[1][\value{footnote}]{\footnotemark[#1]}

\usepackage[breaklinks=true,
pagebackref, 
pdfstartview=FitH,pdfpagemode=UseNone,colorlinks=true,citecolor=blue,linkcolor=blue]{hyperref}

\usepackage[mathscr]{euscript}
\usepackage{bbm}
\usepackage{color,soul}

\usepackage{nicefrac}
\usepackage[nameinlink,capitalise]{cleveref}

\usepackage{scalerel}
\SetLabelAlign{fixedwidth}{\hss\llap{\makebox[-0.5em][l]{#1}}}
\usepackage{relsize}

\usepackage[english]{babel}
\usepackage[utf8]{inputenc}
\usepackage{mathtools}
\usepackage{csquotes}

\newcommand{\HS}{\operatorname{HS}}

\newcommand{\E}{\mathop{\mathbb{E}}}
\newcommand{\F}{\mathbb{F}}

\newcommand{\reals}{\mathbb{R}}

\newcommand{\complex}{\mathbb{C}}

\makeatletter
\newcommand{\stackalign}[1]{
  \vcenter{
    \Let@ \restore@math@cr \default@tag
    \baselineskip\fontdimen10 \scriptfont\tw@
    \advance\baselineskip\fontdimen12 \scriptfont\tw@
    \lineskip\thr@@\fontdimen8 \scriptfont\thr@@
    \lineskiplimit\lineskip
    \ialign{\hfil$\m@th\scriptstyle##$&$\m@th\scriptstyle{}##$\crcr
      #1\crcr
    }
  }
}

\newtheorem{lemma}{Lemma}
\newtheorem{cor}[lemma]{Corollary}

\newtheorem{clm}[lemma]{Claim}

\newtheorem{defn}[lemma]{Definition}
\newtheorem{thm}[lemma]{Theorem}
\newtheorem{prop}[lemma]{Proposition}

\allowdisplaybreaks

\title{Mixing of 3-term progressions in Quasirandom Groups}
\author{
    Amey Bhangale
    \thanks{University of California, Riverside, CA, USA. Email : \texttt{ameyb@ucr.edu, sourya.roy@email.ucr.edu}}
    \and 
    Prahladh Harsha
    \thanks{Tata Institute of Fundamental Research, Mumbai, India. Email : \texttt{prahladh@tifr.res.in}}
    \and
    Sourya Roy
    \samethanks[1]
    }

\begin{document}
\date{}
\maketitle

\begin{abstract}

In this note, we show the mixing of three-term progressions $(x, xg, xg^2)$ in every finite quasirandom groups, fully answering a question of Gowers. More precisely, we show that for any $D$-quasirandom group $G$ and any three sets $A_1, A_2, A_3 \subset G$, we have
\[ \left|\Pr_{x,y\sim G}\left[ x \in A_1, xy \in A_2, xy^2 \in A_3\right] - \prod_{i=1}^3 \Pr_{x\sim G}\left[x \in A_i\right] \right| \leq \left(\frac{2}{\sqrt{D}}\right)^{\nicefrac14}.\] 
Prior to this, Tao answered this question when the underlying quasirandom group is $\mathrm{SL}_{d}(\F_q)$. Subsequently, Peluse extended the result to all nonabelian finite {\em simple} groups. In this work, we show that a slight modification of Peluse's argument is sufficient to fully resolve Gower's quasirandom conjecture for 3-term progressions. Surprisingly, unlike the proofs of Tao and Peluse, our proof is elementary and only uses basic facts from nonabelian Fourier analysis.    
\end{abstract}
\noindent

\section{Introduction}
In this note, we revisit a conjecture by Gowers~\cite{Gowers2008} about mixing of three term arithmetic progressions in quasirandom finite groups. Gowers initiated the study of quasirandom groups while refuting a conjecture of Babai and S{\'{o}}s~\cite{BabaiS1985} regarding the size of the largest product-free set in a given finite group. A finite group is said to be $D$-quasirandom for a positive integer $D>1$ if all its non-trivial irreducible representations are at least $D$-dimensional. The quasirandomness property of groups can be used to show that certain "objects" related to the group "mix" well. For instance, the quasirandomness of the group $\operatorname{PSL}_2(\F_q)$ can be used to give an alternate (and weaker) proof~\cite{DavidoffSV-ramanujan} that the Ramanujan graphs of Lubotzky, Philips and Sarnak~\cite{LubotzkyPS1988} are expanders. Bourgain and Gambard~\cite{BourgainG2008} used quasirandomness to prove that certain other Cayley graphs were expanders.

Gowers proved that for any $D$-quasirandom group $G$ and any three subsets $A,B,C \subset G$ satisfying  $|A|\cdot |B|\cdot |C| \geq |G|^3/D$, there exist $a \in A, b\in B, c\in C$ such that $ab=c$. More generally, he proved that the number of such triples $(a,b,c) \in A \times B \times C$ such that $ab=c$ is at least $(1-\eta)|A|\cdot |B|\cdot |C| /|G|$ provided $|A|\cdot |B|\cdot |C| \geq |G|^3/\eta^2D$. In other words the set of triples of the form $(a,b,ab)$ mix well in a quasirandom group. Gowers' proof of this result was the inspiration and the first step towards the recent optimal inapproximability result for satisfiable $k$LIN over nonabelian groups~\cite{BhangaleK2021}. After proving the well-mixing of triples of the form $(a,b,ab)$ in quasirandom groups, Gowers conjectured a similar statement for triples of the form $(x,xg,xg^2)$. More precisely, he conjectured the following statement: Let $G$ be a ${D}$-quasirandom group and $f_1,f_2,f_3: G\rightarrow \mathbb{C}$ such that $\| f_i\|_{\infty} \leq 1$, then \begin{equation}\label{eq:conjecture}\Big |\E_{x,y\sim G} \big[f_1(x)f_2(xy)f_3(xy^2) \big] - \prod_{i=1,2,3} \E_{x\sim G}\big[f_i(x) \big] \Big | = o_D(1),\end{equation}
where the expression $o_D(1)$ goes to zero as $D$ increases. The conjecture can be naturally extended to $k$-term arithmetic progressions and product of $k$ functions for $k>3$. However, in this note we will focus on the three term case. 

For the specific case of 3-term progressions, Tao~\cite{Tao2013} proved the conjecture for the group $\operatorname{SL}_d(\F_q)$ for bounded $d$ using algebraic geometric machinery. In particular, he proved that the right-hand side expression in \cref{eq:conjecture} can be bounded by $O(1/q^{\nicefrac18})$ when $d=2$ and $O_d(1/q^{\nicefrac14})$ for larger $d$. Tao's approach relied on algebraic geometry and was not amenable to other quasirandom groups. Later, Peluse~\cite{Peluse2018} proved the conjecture for all nonabelian finite simple groups. She used basic facts from nonabelian Fourier analysis
to prove that the right-hand side expression in \cref{eq:conjecture} can be bounded by $\sum_{1 \neq \rho \in \hat{G}} 1/d_\rho$ where $\hat{G}$ represents the set of irreducible unitary representation of $G$ and $d_\rho$ the dimension of the irreducible representation $\rho$. This latter quantity is the \emph{Witten zeta function} $\zeta_G$ of the group $G$ minus one and can be bounded for \emph{simple} finite quasirandom groups using a result due to Liebeck and Shalev~\cite{LiebeckS2004}. 

In this paper, we show that a slight variation of Peluse's argument can be used to prove the conjecture for \emph{all quasirandom groups} with \emph{better} error parameters. More surprisingly, the proof stays completely elementary and short. Specifically, we prove the following statement: 
\begin{thm}\label{thm:main}
Let $G$ be a ${D}$-quasirandom finite group, i.e, its all  non-trivial irreducible representations are at least $D$-dimensional. Let $f_1,f_2,f_3: G\rightarrow \mathbb{C}$ such that $\| f_i\|_{\infty} \leq 1$ then \[\Big|\E_{x,y\sim G} \big[f_1(x)f_2(xy)f_3(xy^2) \big] - \prod_{i=1,2,3} \E_{x\sim G}\big[f_i(x) \big] \Big | \leq  \left(\frac{2}{\sqrt{D}}\right)^{\frac{1}{4}}.\]
\end{thm}


\section{Preliminaries}

We begin by recalling some basic representation theory and nonabelian Fourier analysis. See the monograph by Diaconis~\cite[Chapter 2]{Diaconis1998} for a more detailed treatment (with proofs). 

We will be working with a finite group $G$ and complex-valued functions $f\colon G \to \complex$ on $G$. All expectations will be with respect to the uniform distribution on $G$.  The \emph{convolution} between two function $f,h\colon G\to \complex$, denoted by $f*h$, is defined as follows:
\[ (f*h)(x) :=\E_{y} [f(xy^{-1})h(y)] .\] 
For any $p \geq 1$, the $p$-norm of any function $f\colon G \to \complex$ is defined as 
\[\|f\|^p_p := \E_{x}[|f(x)|^p] .\]

For any element $g \in G$, the \emph{conjugacy class of $g$}, denoted by $C(g)$, refers to the set $\{ x^{-1}gx | x \in G\}$. Observe that the conjugacy classes form a partition of the group $G$. A function $f\colon G \to \complex$ is said to be a \emph{class function} if it is constant on conjugacy classes.  

For any $b\in G$  we use $\Delta_b f(x) :=f(x)\cdot f(xb) $. For any set $S\subset G$,  $\mu_S \colon G \to \reals$ denotes the scaled density function $\frac{|G|}{|S|}\mathbbm{1}_S$. The scaling ensures that $\E_x[\mu_S(x)] =1$. 

Given a complex vector space $V$, we denote the vector space of linear operators on $V$ by $\operatorname{End}(V)$. This space is endowed with the following inner product and norm (usually referred to as the \emph{Hilbert-Schmidt} norm): \[ \text{For } A, B \in \operatorname{End}(V), \quad \langle A,B\rangle_{\HS} := \operatorname{Trace}(A^*B) \quad \text{ and } \quad \|A \|_{\HS}^2 := \langle A, A \rangle_{\HS} = \operatorname{Trace}(A^*A) . \]
This norm is known to be submultiplicative (i.e, $\|AB \|_{\HS}\leq \|A\|_{\HS}\cdot \|B\|_{\HS}$).

\paragraph{Representations and Characters:} A \emph{representation} $\rho \colon G \to \operatorname{End}(V)$ is a homomorphism from $G$ to the set of linear operators on $V$ for some finite-dimensional vector space $V$ over $\complex$, i.e., for all $x, y \in G$, we have $\rho(xy) = \rho(x) \rho(y)$. The dimension of the representation $\rho$, denoted by $d_\rho$, is the dimension of the underlying $\complex$-vector space $V$. The \emph{character} of a representation $\rho$, denoted by $\chi_\rho \colon G \to \complex$, is defined as $\chi_\rho(x) := \operatorname{Trace}(\rho(x))$.

The representation $1 \colon G \to \complex$ satisfying $1(x) = 1$ for all $x \in G$ is the \emph{trivial} representation. A representation $\rho \colon G \to \operatorname{End}(V)$ is said to \emph{reducible} if there exists a non-trivial subpsace $W \subset V$ such that for all $x \in G$, we have $\rho(x) W \subset W$. A representation is said to be \emph{irreducible} otherwise. The set of all irreducible representations of $G$ (upto equivalences) is denoted by $\hat{G}$. 

For every representation $\rho \colon G \to \operatorname{End}(V)$, there exists an inner product $\langle , \rangle_V$ over $V$ such that every $\rho(x)$ is unitary (i.e, $\langle \rho(x) u, \rho(x) v\rangle_V = \langle u, v \rangle_V$ for all $u,v \in V$ and $x \in G$). Hence, we might wlog. assume that all the representations we are considering are unitary.

The following are some well-known facts about representations and characters.
\begin{prop}\label{prop:rep}
\begin{enumerate}
\item The group $G$ is abelian iff $d_\rho =1$ for every irreducible representation $\rho$ in $\hat{G}$.
\item For any finite group $G$, $\sum_{\rho \in \hat{G}} d^2_\rho = |G|$.
\item {[orthogonality of characters]} For any $\rho,\rho'\in \hat{G}$ we have: $\E_{x} \Big[ \chi_{\rho}(x) \overline{\chi_{\rho'}(x)} \Big]= \mathbb{1}[\rho=\rho']$. 
\end{enumerate}
\end{prop}

\begin{defn}[quasirandom groups] A nonabelian group $G$ is said to be \emph{$D$-quasirandom} for some positive integer $D>1$ if all its non-trivial irreducible representations $\rho$ satisfy $d_\rho \geq D$.
\end{defn}

\paragraph{Nonabelian Fourier analysis:} Given a function $f \colon G \to \complex$ and an irreducible representation $\rho \in \hat{G}$, the Fourier transform is defined as follows:
\[ \hat{f}(\rho) := \E_{x} [f(x) \rho(x)] . \]
The following proposition summarizes the basic properties of Fourier transform that we will need.

\begin{prop}\label{prop:fourier} For any $f,h \colon G \to \complex$, we have the following
\begin{enumerate}
    \item {[Fourier transform of trivial representation]} \[\hat{f}(1) = \E_x[f(x)].\]
    \item {[Convolution]} \[\widehat{f*h}(\rho)=\hat{f}(\rho)\cdot \hat{h}(\rho).\]
    \item {[Fourier inversion formula]} \[f(x)= \sum_{ \rho \in \hat{G}}  d_{\rho} \cdot \langle\hat{f}(\rho), \rho(x)\rangle_{\HS}. \]
    \item {[Parseval’s identity]} \[ \|f \|^2_2= \sum_{\rho \in \hat{G}} d_{\rho} \cdot \|\hat{f}(\rho) \|^2_{\HS} .\]
    \item {[Fourier transfrom of class functions]} For any class function $f \colon G \to \complex$, the Fourier transform satisfies \[\hat{f}(\rho) = c \cdot I_{d_\rho}\] for some constant $c = c(f,\rho)\in \complex$. In other words, the Fourier transform is a scaling of the Identity operator $I_{d_\rho}$.
\end{enumerate}
\end{prop}

The following claim (also used by Peluse \cite{Peluse2018}) observes that the scaled density function $\mu_{gC(g)}$ has a very simple Fourier transform since it is a translate of the class function $\mu_{C(g)}$
\begin{clm}
\label{fact:fc_mu}
 For any $g\in G$ and  $\rho\in \hat{G}$ we have: \[\hat{\mu}_{gC(g)}(\rho)= \frac{\chi_{\rho}(g)}{d_{\rho}} \cdot \rho(g)\] where $C(g)$ refers to the conjugacy class of $g$. Moreover, $\|\hat{\mu}_{gC(g)}  \|^2_{\HS}=\frac{|\chi_{\rho}(g)|^2}{d_{\rho}}$
\end{clm}
\begin{proof} We begin by observing that 
\begin{align*}
\hat{\mu}_{gC(g)}(\rho)&=\E_{x}\left[\mu_{gC(g)}(x)\cdot \rho(x)\right]
\\&= \E_{x}\left[\mu_{gC(g)}(gx)\cdot \rho(gx)\right] \\
&= \E_{x}\left[\mu_{gC(g)}(gx)\cdot \rho(g) \cdot \rho(x)\right]
\\&= \rho(g)\cdot \E_{x}\left[\mu_{C(g)}(x)\cdot \rho(x)\right]\\
& = \rho(g)\cdot \hat{\mu}_{C(g)}(\rho).
\end{align*}
On the other hand, as $\mu_{C(g)}$ is a class function, we have  $\hat{\mu}_{C(g)}(\rho)=c \cdot I_{d_{\rho}} $ for some constant $c\in \complex$. The constant $c$ can be determined by taking trace on either side of $c\dot I_{d_\rho} = \hat{\mu}_{C(g)} = \E_x[ \mu_{C(g)} (x) \cdot \rho(x)]$ and noting that $\operatorname{Trace}(\rho(x))=\chi_{\rho}(g)$ as follows:
\begin{align*}
    c \cdot d_\rho = \E_x\left[\mu_{C(g)}(x) \cdot \chi_\rho(x) \right] = \E_x\left[\mu_{C(g)}(x) \right] \cdot \chi_\rho(g) = \chi_\rho(g).
\end{align*}  
Hence, $c=\frac{\chi_{\rho}(g)}{d_{\rho}}$ and $\hat{\mu}_{gC(g)}=\frac{\chi_{\rho}(g)}{d_{\rho}} \cdot \rho(g)$. Lastly we have,
\begin{align*}
 \|\hat{\mu}_{gC(g)}\|^2_{\HS} &=   \left\| \frac{\chi_{\rho}(g)}{d_{\rho}}\cdot \rho(g) \right\|^2_{\HS}
 \\&=\frac{|\chi_{\rho}(g)|^2}{d^2_{\rho}}\cdot \operatorname{Trace}\left(\rho(g)^*\cdot \rho(g)\right)\\
 &=\frac{|\chi_{\rho}(g)|^2}{d^2_{\rho}}\cdot d_{\rho} \tag*{(\textrm{By unitariness of $\rho(g)$})}
 \\&=\frac{|\chi_{\rho}(g)|^2}{d_{\rho}}.\tag*\qedhere
\end{align*}
\end{proof}
The key property of $D$-quasirandom groups that we will be using is the following inequality due to Babai, Nikolov and Pyber, the proof of which we provide for the sake of completeness.
\begin{lemma}[\cite{BabaiNP2008}]\label{fact:BNP}
If $G$ is a $D$-quasirandom group and $f_1,f_2\colon G\to \complex$ such that either $f_1$ or $f_2$ is mean zero  then 
\[ \|f_1*f_2 \|_2 \leq \frac{1}{\sqrt{D}} \cdot \|f_1\|_2\cdot \|f_2 \|_2. \]
\end{lemma}
\begin{proof}
\begin{align*}
    \|f_1 * f_2 \|^2 
    & = \sum_{\rho \in \hat{G}} d_\rho \|\widehat{f_1 * f_2}(\rho)\|^2_{\HS}\\ 
    & = \sum_{\rho \in \hat{G}} d_\rho \|\hat{f_1}(\rho)\cdot \hat{f_2}(\rho)\|^2_{\HS}\\
    & \leq \sum_{\rho \in \hat{G}} d_\rho\|\hat{f_1}(\rho)\|^2_{\HS} \cdot \|\hat{f_2}(\rho)\|^2_{\HS} \tag*{(By submultiplicativity of norm)}\\
    &= \sum_{1 \neq \rho \in \hat{G}} d_\rho \|\hat{f_1}(\rho)\|^2_{\HS} \cdot \|\hat{f_2}(\rho)\|^2_{\HS} \tag*{(By mean zeroness)}\\
    & \leq \frac1D \cdot \sum_{1\neq \rho \in \hat{G}} d^2_\rho \|\hat{f_1}(\rho)\|^2_{\HS} \cdot \|\hat{f_2}(\rho)\|^2_{\HS} \tag*{(By $D$-quasirandomness)}\\
    & \leq \frac1D \left( \sum_{1\neq \rho \in \hat{G}} d_\rho \|\hat{f_1}(\rho)\|^2_{\HS}\right) \cdot \left( \sum_{1\neq \rho \in \hat{G}} d_\rho \|\hat{f_2}(\rho)\|^2_{\HS}\right)\\
    & \leq \frac1D \cdot \|f_1\|_2^2 \cdot \|f_2\|_2^2\;. \tag*\qedhere
\end{align*}
\end{proof}

The following is a simple corrollary of \cref{fact:BNP}.
\begin{cor}
\label{fact:der}
If $G$ is  $D$-quasirandom; $f\colon G \to \complex$ has zero mean and $\|f \|_{\infty} \leq 1$ then \[\E_{b}\big[|\E_{x} \Delta_bf(x)  |\big] \leq \frac{1}{\sqrt{D}}.\]
\end{cor}

\begin{proof}
Let $f'(x) := f(x^{-1})$. We have, 
\begin{align*}
\E_{b}\big[|\E_{x} \Delta_bf(x)  |\big]&= \E_{b}\big[|\E_{x} f(x)f(xb)  |\big]
\\&= \E_{b}\Big[\big|\E_{x} f'(x^{-1})f(xb) \big |\Big] 
\\&= \E_{b}\Big[|f'*f(b)| \Big] 
\\& \leq \E_{b}\Big[|f'*f(b)|^2 \Big]^{1/2} \tag*{(\textrm{By Cauchy-Schwarz inequality})}
\\& = \|f'*f\|_2
\\& \leq \frac{1}{\sqrt{D}} \cdot \|f'\|_2 \cdot \|f\|_2 \tag*{(By \cref{fact:BNP})}
\\&\leq  \frac{1}{\sqrt{D}} .\tag*{(Since $\|f\|_2 \leq \|f\|_{\infty} \leq 1$).}
\end{align*}
\end{proof}

\section{Proof of {\cref{thm:main}}}

The following proposition is where we deviate from Peluse's proof~\cite{Peluse2018}. We give an elementary proof for \emph{every} quasirandom group while Peluse proved the same result for \emph{simple} finite groups using the result of Liebeck and Shalev~\cite{LiebeckS2004} to bound the Witten zeta function $\zeta_G$ for \emph{simple} finite groups. 
\begin{prop}
\label{lemma:main}
Let $G$ be a $D$-quasirandom group. Let $f\colon G \to \complex$ such that $\|f\|_{\infty} \leq 1$, $\E[f]=0$ and $f_{b}$ is the mean zero component of the function $\Delta_{b} f$ (i.e., $f_b(x) = \Delta_b f(x) - \E_x[\Delta_b f(x)]$).  Then 
\[\E_{g,b} \bigg[ \Big | \E_x \big[ \Delta_b f(x)  \cdot (f_{g^{-1}bg} * \mu_{g^{-1}C(g^{-1})})(x) \big]  \Big| \bigg]\leq \frac{1}{\sqrt{D}}\;. \]
\end{prop}

\begin{proof} Let us denote the expression on the L.H.S. as $\Gamma$. We use simple manipulations and previously stated facts to simplify the expression.
\begin{align*}
  \Gamma^2 & \leq \E_{g,b}\bigg[\|\Delta_b f \|_2\cdot \| (f_{g^{-1}bg} * \mu_{g^{-1}C(g^{-1})} \|_2 \bigg]^2 \tag*{\text{(By Cauchy-Schwarz inequality)}}
     \\ & \leq 
     \E_{g,b}\bigg[  \|f_{g^{-1}bg} * \mu_{g^{-1}C(g^{-1})} \|_2 \bigg]^2 \tag*{\text{ (Since} $\|\Delta_b f \|_2 \leq 1$)}
     \\ & \leq
     \E_{g,b}\bigg[ \|f_{g^{-1}bg} * \mu_{g^{-1}C(g^{-1})} \|^2_2 \bigg] \tag*{\text{(By Cauchy Schwarz inequality)}}
     \\ & =
     \E_{g,b}\bigg[ \sum_{1\neq \rho \in \hat{G}} d_{\rho} \cdot \|\hat{f}_{g^{-1}bg}(\rho) \cdot \hat{\mu}_{g^{-1}C(g^{-1})}(\rho)\|^2_{\HS} \bigg]\tag*{\text{(By Parseval's identity \& $\hat{f}_{g^{-1}bg}(1)=0$ )}}
    \\ & \leq  \E_{g,b}\bigg[ \sum_{1\neq \rho \in \hat{G}} d_{\rho} \cdot \|\hat{f}_{gbg^{-1}}(\rho)\|^2_{\HS}\cdot  \| \hat{\mu}_{g^{-1}C(g^{-1})}(\rho) \|^2_{\HS} \bigg] \tag*{\text{(By submultiplicativity of norm)}}
    \\& =  \E_{g,b}\bigg[ \sum_{1\neq \rho \in \hat{G}}   \|\hat{f}_{g^{-1}bg}(\rho)\|^2_{\HS} \cdot |\chi_{\rho}(g)|^2 \bigg] \tag*{\text{(By \cref{fact:fc_mu})}}
    \\&
    =     \sum_{1\neq \rho \in \hat{G}} \E_{g}\Big[|\chi_{\rho}(g)|^2 \cdot \E_{b}\Big[\big \|  \hat{f}_{gbg^{-1}}(\rho)   \big \|^2_{\HS}\Big] \Big].
    \end{align*}
    
    Now using the fact that $gbg^{-1}$ is uniformly distributed in $G$ for a fixed $g$ and a uniformly random $b$ in $G$, we can simply the above expression as follows.
    
    \begin{align*}
  \Gamma^2 &
    \leq    \sum_{1\neq \rho \in \hat{G}}   \E_{g}\Big[|\chi_{\rho}(g)|^2 \cdot \E_{b}\Big[\big \|  \hat{f}_{b}(\rho)   \big \|^2_{\HS}\Big] \Big]
    \\&
    = \sum_{1\neq \rho \in \hat{G}}   \E_{b}\Big[\big \|  \hat{f}_{b}(\rho)   \big \|^2_{\HS}\Big] \cdot \E_{g}\Big[|\chi_{\rho}(g)|^2 \Big]\\
    &
    =     \sum_{1\neq \rho \in \hat{G}}   \E_{b}\Big[\big \|  \hat{f}_{b}(\rho)   \big \|^2_{\HS}\Big] \tag*{\text{(By orthogonality of $\chi_\rho$)}}
    \\&
    =    \E_{b}  \Big[\sum_{1\neq \rho \in \hat{G}}   \big \|  \hat{f}_{b}(\rho)   \big \|^2_{\HS}\Big].
    \end{align*}
    
    Finally, we use the fact that all the terms in the summation are non-negative and the group $G$ is a $D$-quasirandom group. 
    \begin{align*}
    \Gamma^2&   \leq  \frac{1}{D} \cdot \E_{b}  \Big[\sum_{1\neq \rho \in \hat{G}} d_{\rho} \cdot\big \|  \hat{f}_{b}(\rho)   \big \|^2_{\HS}\Big] 
    \\& = \frac{1}{D} \cdot \E_{b}  \Big[ \|f_b\|^2_{2} \Big] \tag*{\text{(By Parseval's identity)}}
    \\& \leq  \frac{1}{D}, \tag*{\text{(Because $\|f_b \|^2_2 \leq 1$).}}
    \end{align*}
    The proof of this lemma is similar to the proof of the BNP inequality (\cref{fact:BNP}). The key difference being that we have a complete characterization of the Fourier transform of $\mu_{gC(g)}$ from \cref{fact:fc_mu} which we use to give a sharper bound.
\end{proof}    

We are now ready to prove the main \cref{thm:main}.
This part of the proof is similar to the corresponding expression that appears in the paper of Peluse~\cite{Peluse2018}, which is in turn inspired by Tao's adaptation of Gowers' repeated Cauchy-Schwarzing trick to the nonebelian setting. We, however, present the entire proof for the sake of completeness.

\begin{proof}[Proof of {\cref{thm:main}}]
Let us denote the L.H.S. of the expression by $\Theta_{f_1, f_2, f_3}$. Without loss of generality we assume  $\E [f_3]=0$. Now we have, 
\begin{align*}
   \Theta^4_{f_1, f_2, f_3}&=
   \Big|\E_{x,y} \big[f_1(x)f_2(xy)f_3(xy^2) \big]\Big|^4
   \\&= \Big|\E_{x,z} \big[f_1(xz^{-1})f_2(x)f_3(xz) \big]  \Big |^4 \tag*{\text{ (Change of variables}: $x\leftarrow xy, z\leftarrow y$)}
   \\& \leq \Big|\E_{x,z_1,z_2} \big[f_1(xz_1^{-1})f_1(xz_2^{-1})f_3(xz_1) f_3(xz_2) \big]  \Big |^2 \tag*{(\textrm{Cauchy-Schwarz over $x$; $\|f_2\|_{\infty}=1$ and expansion} )}
   \\& =\Big|\E_{y,z,a} \big[f_1(y)f_1(ya)f_3(yz^2) f_3(yza^{-1}z) \big]  \Big |^2 \tag*{\text{ (Change of variables}: $y\leftarrow xz^{-1}_1, z \leftarrow z_1, a \leftarrow z_1z^{-1}_2$)}
   \\&= \Big|\E_{y,z,a} \big[\Delta_{a}f_1(y) \cdot \Delta_{z^{-1}a^{-1}z}~f_3(yz^2)  \big]  \Big |^2  
   \\& \leq \Big|\E_{y,a,z_1,z_2} \big[  \Delta_{z_1^{-1}a^{-1}z_1}~f_3(yz_1^2) \cdot \Delta_{z_2^{-1}a^{-1}z_2}~f_3(yz_2^2)  \big]  \Big |, 
   \tag*{(\textrm{Cauchy-Schwarz over $y,a$; $\|f_1\|_{\infty}\leq 1$} ).}
 \end{align*}
 Now, using the following change of variables, $z \leftarrow z_1,~ x\leftarrow yz^2_1,~  b\leftarrow z_1^{-1}a^{-1}z_1,~ g \leftarrow z_1^{-1}z_2 $ , we get
 \begin{align*}
  \Theta^4_{f_1, f_2, f_3}&\leq
   \Big|\E_{x,b,z,g} \big[  \Delta_{b}~f_3(x) \cdot \Delta_{g^{-1}bg}~f_3(xz^{-1}gzg)  \big]  \Big | \\
   &=  \bigg|\E_{x,b,g} \Big[  \Delta_{b}~f_3(x) \cdot \E_{z}[\Delta_{g^{-1}bg}~f_3(xz^{-1}gzg)]  \Big]  \bigg |
   \\&= \bigg|\E_{x,b,g} \Big[  \Delta_{b}~f_3(x) \cdot \E_{a}[\Delta_{g^{-1}bg}~f_3(xa^{-1})\cdot \frac{|G|}{|C(g^{-1})|} 1_{g^{-1}C(g^{-1})} (a) ]\Big]  \bigg | 
   \\&=\bigg|\E_{x,b,g} \Big[  \Delta_{b}~f_3(x) \cdot \E_{a}[\Delta_{g^{-1}bg}~f_3(xa^{-1})\cdot \mu_{g^{-1}C(g^{-1})} (a)] \Big]  \bigg |
   \\&= \bigg|\E_{x,b,g} \Big[  \Delta_{b}~f_3(x) \cdot \Delta_{g^{-1}bg}~f_3* \mu_{g^{-1}C(g^{-1})}(x) \Big]  \bigg |.
   \end{align*}
   We now separate the function $\Delta_{g^{-1}bg}~f_3$ from its the mean zero part as follows: Let $\Delta_{g^{-1}bg}~f_3 = f'_{g^{-1}bg} + f_{g^{-1}bg}$ where $f'_{g^{-1}bg}=\E_{x}[\Delta_{g^{-1}bg}~f_3(x)]$ and $f_{g^{-1}bg} (x) = \Delta_{g^{-1}bg}~f_3(x) - f'_{g^{-1}bg}$. 
   \begin{align*}
   \Theta^4_{f_1, f_2, f_3}& \leq \bigg|\E_{x,b,g} \Big[  \Delta_{b}~f_3(x) \cdot (f_{g^{-1}bg}+f'_{g^{-1}bg})* \mu_{g^{-1}C(g^{-1})}(x) \Big]  \bigg |
   \\& \leq  \E_{b,g} \bigg[\Big|\E_{x} \big[  \Delta_{b}~f_3(x) \cdot f_{g^{-1}bg}* \mu_{g^{-1}C(g^{-1})}(x) \big]  \Big | \bigg]\\
   &\quad\quad\quad\quad\quad\quad\quad\quad\quad +    \E_{b,g} \bigg[\Big|\E_{x} \big[  \Delta_{b}~f_3(x) \cdot f'_{g^{-1}bg}* \mu_{g^{-1}C(g^{-1})}(x) \big]  \Big | \bigg]
   \\& \leq \frac{1}{\sqrt{D}} +  \E_{b,g} \Big[\big|\E_{x} \big[  \Delta_{b}~f_3(x)  \big]  \big | \cdot \|f'_{g^{-1}bg}* \mu_{g^{-1}C(g^{-1})} \|_{\infty}\Big] \tag*{(\textrm{Using \cref{lemma:main} to bound the first expectation})}
  \\& = \frac{1}{\sqrt{D}} +  \E_{b,g} \Big[\big|\E_{x} \big[  \Delta_{b}~f_3(x)  \big]  \big | \cdot |f'_{g^{-1}bg}|\Big] 
    \\& \leq \frac{1}{\sqrt{D}} +  \E_{b} \Big[\big|\E_{x} \big[  \Delta_{b}~f_3(x)  \big]  \big | \Big] \tag*{(Using $|f'_{g^{-1}bg}| \leq 1$)}
   \\& \leq \frac{2}{\sqrt{D}}\;,\tag*{(\textrm{By \cref{fact:der} and $\|f_3 \|_{\infty} \leq 1$}).}
\end{align*}
\end{proof}

{\small 
  \bibliographystyle{prahladhurl}
  \bibliography{3apquasi-bib}
  }

\end{document}